\documentclass[a4paper,12pt]{amsart}
\usepackage{amsmath,amsfonts,amssymb,amsthm,amscd}
\usepackage{latexsym,graphicx}
\usepackage[matrix,arrow,ps,color,line,curve,frame]{xy}
\usepackage[usenames,dvipsnames]{color}

\usepackage{xcolor}

\renewcommand{\[}{\begin{equation}}
\renewcommand{\]}{\end{equation}}

\newtheorem{proposition}{\sc Proposition}[section]
\newtheorem{lemma}[proposition]{\sc Lemma}
\newtheorem{corollary}[proposition]{\sc Corollary}
\newtheorem{theorem}[proposition]{\sc Theorem}
\newtheorem{conjecture}[proposition]{\sc Conjecture}
\newtheorem{definition}[proposition]{\sc Definition}
\theoremstyle{definition}

\theoremstyle{remark}
\newtheorem{remark}[proposition]{\sc Remark}

\newcommand{\id}{\operatorname{id}}

\renewcommand{\ker}{\operatorname{Ker}}

\def\veps{{\varepsilon}}

\newcommand{\ot}{\otimes}
\renewcommand{\phi}{\varphi}
\renewcommand{\epsilon}{\varepsilon}

\newcommand{\cO}{\mathcal{O}}

\def\P{\mathcal{P}}
\def\C{{\mathbb C}}

\def\Z{{\mathbb Z}}
\def\z2{{\mathbb Z}/2{\mathbb Z}}

\def\id{{\rm id}}

\usepackage[T1]{fontenc}
\usepackage{stmaryrd}

\newcommand{\ev}{{\mathrm{ev}}}
\newcommand{\bp}{\begin{proof}}
\newcommand{\ep}{\end{proof}}

\newcommand{\eps}{\varepsilon}

\newcommand{\CC}{\mathcal{C}}
\newcommand{\SSS}{\mathcal{S}}

\newcommand{\mfg}{\mathfrak{g}}

\newcommand{\Pol}{\mathcal{O}}

\newcommand{\cP}{\mathcal{P}}

\begin{document}
\baselineskip=15pt
\parindent=5mm

\author{Ludwik D\k abrowski}
\address{SISSA (Scuola Internazionale Superiore di Studi Avanzati)\\
Via Bonomea 265, 34136 Trieste, Italy
}
\email{dabrow@sissa.it}

\author{Piotr M.~Hajac}
\address{Institytut Matematyczny, Polska Akademia Nauk\\
ul.\ \'Sniadeckich 8, 00-656 Warszawa, Poland\\
}
\email{pmh@impan.pl}

\author{Sergey Neshveyev}
\address{Department of Mathematics, University of Oslo\\ P.O. Box 105, 3 Blindern, 0316 Oslo,
Norway\\
}
\email{sergeyn@math.uio.no}

\title[Noncommutative Borsuk-Ulam-type conjectures revisited]
{\large Noncommutative Borsuk-Ulam-type\\\vspace*{3mm} conjectures revisited}
\vspace*{5mm}
\begin{abstract}
Let $H$ be the C*-algebra of a non-trivial compact quantum group acting freely on a unital C*-algebra $A$.
It was recently conjectured that there does not exist an equivariant
 $*$-homomorphism from $A$ (type-I case) or $H$ (type-II case)
to the equivariant noncommutative join C*-algebra $A\circledast^\delta H$.
When $A$ is the C*-algebra of functions on a sphere, and $H$ is the C*-algebra of functions on $\Z/2\Z$
acting antipodally on the sphere, then the conjecture of type I becomes the celebrated Borsuk-Ulam theorem.
Following recent work of Passer, we prove the conjecture of type I for compact quantum groups admitting a non-trivial torsion character.
Next, we prove that, if a compact quantum group admits a representation whose \mbox{$K_1$-class} is non-trivial and $A$ admits a character,
then a stronger version of the type-II conjecture holds:  the
finitely generated projective module associated with $A\circledast^\delta H$ via this representation is not
stably free.
In particular, we apply this result to the
$q$-deformations of  compact connected semisimple Lie groups and to the reduced group C*-algebras of free groups on $n>1$
generators.
\end{abstract}
\maketitle
{\tableofcontents}
\clearpage

\section{Introduction}
\noindent
The goal of this paper is to prove, under some additional assumptions, both types  of the conjecture stated
in~\cite[Conjecture~2.3]{bdh15}.
We will also derive some consequences of  the type~I conjecture,
and a K-theoretic strengthening of  the type II conjecture. 

\subsection{Around the classical Borsuk-Ulam theorem}
The Borsuk-Ulam theorem \cite{b-k33}
is a fundamental theorem of topology concerning spheres in Euclidean spaces.
It can be phrased in the following three equivalent ways:

\begin{theorem}[Borsuk-Ulam]
For any natural number $n$,
if $f\colon S^{n+1}
\to\mathbb{R}^{n+1}$ is continuous, then there exists a pair $(p,-p)$ of antipodal points on
$S^{n+1}$ such that $f(p)=f(-p)$.
\end{theorem}
\begin{theorem}[equivariant formulation]
For any natural number $n$,
there  does \emph{not} exist a $\mathbb{Z}/2\mathbb{Z}$-equivariant
continuous  map $\widetilde{f}\colon S^{n+1}\to S^{n}$.
\end{theorem}
\begin{theorem}[join formulation]
For any natural number $n$,
there  does \emph{not} exist a $\mathbb{Z}/2\mathbb{Z}$-equivariant
continuous  map $\widetilde{f}\colon S^{n}*\mathbb{Z}/2\mathbb{Z}\to S^{n}$.
\end{theorem}

The third formulation of the theorem lends itself to a natural and wide generalization, which was
proposed in
\cite[Conjecture~2.2]{bdh15}(cf.\ \cite[Theorem~6.2.5]{m-j03}).
It is worth mentioning that \cite[Conjecture~2.2]{bdh15} without the assumption of the existence of non-trivial torsion,
or local triviality, implies a weaker version of the celebrated Hilbert-Smith conjecture (see \cite{cdt} for details), and remains wide open.
However, under the assumption that the group
is not torsion free, \cite[Conjecture~2.2]{bdh15} was recently proved by Passer:
\begin{theorem}[Theorem~2.8 in \cite{p-b16}]\label{CP1}
Let $X$ be a compact Hausdorff space equipped with a continuous free action of a
 compact Hausdorff
group~$G$ with a non-trivial torsion element. Then, for the diagonal action of $G$ on the join
$X*G$, there does \emph{not} exist a $G$-equivariant
continuous map $X*G\to X$.
\end{theorem}

Now recall that
 the only contractible compact Hausdorff group $G$ is the trivial group~\cite{h-b79}, and the existence of a $G$-equivariant
continuous map $G*G\to G$ is equivalent to the contractibility of~$G$ (e.g., see \cite{bhms07}). Furthermore,
by taking any $x_0\in X$, we can define a $G$-equivariant continuous map
$G\ni g\mapsto x_0g\in X$, which induces a $G$-equivariant continuous map~\mbox{$G*G\to X*G$}.
Consequently, we obtain:
\begin{theorem}\label{cor2}
Let $X$ be a compact Hausdorff space equipped with a continuous free action of a
 non-trivial compact Hausdorff
group~$G$. Then there does \emph{not} exist a $G$-equivariant
continuous map
\mbox{$
X*G\to G
$}.
\end{theorem}
\noindent
Note that by a similar argument, in case $G$ is not torsion free, Theorem~\ref{cor2} follows from
Theorem~\ref{CP1}.

If we assume even more about~$G$, notably that it is a non-trivial compact connected semisimple Lie group,
then Theorem~\ref{cor2} also
follows from:
\begin{theorem}\label{Kthm}
Let $G$ be a non-trivial compact connected semisimple Lie group.
Then there exists a finite-dimensional representation $V$ of $G$
such that for any compact Hausdorff space $X$  equipped with a
free $G$-action,  the associated vector bundle
$$
(X*G)\overset{G}{\times} V
$$
is \emph{not} stably trivial.
\end{theorem}
\begin{proof}
To begin with, let us show that it suffices to prove the theorem in the special case~$X=G$.
Indeed, let $Y$ and $Z$ be compact Hausdorff $G$-spaces. Assume that the $G$-action on $Z$ is free,
and that there is a continuous $G$-equivariant map $f:Y\to Z$.
It follows that the $G$-action on $Y$ is  free, and that
$f$ induces a continuous map $\bar{f}:Y/G\to Z/G$ between quotient spaces. Now,
 for any finite-dimensional representation $G\to GL(V)$,  we obtain:
\[
\bar{f}^*\colon K^0(Z/G)\ni \Big[Z\overset{G}{\times} V\Big]\longmapsto \Big[Y\overset{G}{\times} V\Big]
\in K^0(Y/G).
\]
Hence, the stable non-triviality of the associated vector bundle $Y\times^G V$ implies
 the stable non-triviality of the associated vector bundle $Z\times^G V$.

Finally, since for any $x_0\in X$ the continuous $G$-equivariant map $G\ni g\mapsto x_0g\in X$
induces a continuous $G$-equivariant map $G*G\to X*G$, the above argument shows that it suffices
to prove the theorem~for $X=G$. Therefore, the following lemma concludes the proof. \ep

\begin{lemma}\label{lem:a-mf65}
Let  $G$ be a non-trivial compact connected semisimple Lie group.
Then there exists a representation of $G$ on a finite-dimensional vector space $V$ such that
$$
[(G*G)\times^G V]\ne[\dim V]\quad\text{in}\quad K^0(\SSS G).
$$
Here $\SSS G$ is the unreduced suspension of~$G$, which we identify with $(G*G)/G$.
\end{lemma}
\bp
Set $E_V:=(G*G)\times^G V$ for brevity. Consider the group
\[
\tilde K^0(\SSS G)=\ker\big(K^0(\SSS G)\longrightarrow K^0(\mathrm{pt})=\Z\big),
\]
where the homomorphism is defined by any point of the connected space $\SSS G$.
 As was already observed in~\cite{a-mf65}, using the description of
$K^1(G)$ in terms of homotopy classes of maps $G\to GL(\infty)$,
 it is not difficult to see that there is an isomorphism
\mbox{$\beta\colon\tilde K^0(\SSS G)\to K^1(G)$} such that
$\beta([\dim V]-[E_V])$ is exactly  $[u]\in K^1(G)$, where
\mbox{$u: G\to GL(V)$} is  the representation defining~$E_V$.

Now, if $G$ is simply connected, then for any fundamental representation $V$ we have
$\beta([\dim V]-[E_V])\neq0$
by~\cite{h-l67}. If $G$ is not simply connected, it is a quotient of its simply connected cover $\widetilde G$
by a finite central subgroup~$\Gamma$. Therefore, to get a desired
representation of the quotient $G=\widetilde G/\Gamma$,
it suffices to consider the tensor power $V^{\otimes d}$ for a fundamental representation
$V$ of $\widetilde G$,
where $d=|\Gamma|$. Then, as
$K^1(\widetilde G)$ has no torsion~\cite{h-l67}, the $K^1(\widetilde G)$-class
\[
\beta\big(\big[\dim V^{\otimes d}\big]-\big[E_{V^{\otimes d}}\big]\big)=
d(\dim V)^{d-1}\beta\big(\big[\dim V\big]-\big[E_{V}\big]\big)
\]
(see~\cite[Section~I.4]{h-l67}) is nonzero,  so 
$\beta\big(\big[\dim V^{\otimes d}\big]-\big[E_{V^{\otimes d}}\big]\big)$ is also nonzero when viewed as a class
 in $K^1(G)$.
\ep

\subsection{Quantum preliminaries}

In this section, we review basic noncommutative concepts used to formulate quantum Borsuk-Ulam-type
statements. We always assume that *-homomorphisms between unital C*-algebras  are  unital.

\subsubsection{Compact quantum groups \cite{w-sl98}}
Recall that a compact quantum semigroup $(H, \Delta)$
consists of a unital C*-algebra $H$ and a {\em coproduct}
$\Delta: H\to H\otimes_{\mathrm{min}} H$, i.e.,
a coassociative $*$-homomorphism. If the coproduct is injective and satisfies also
 the {\em cancellation laws} 
\[\label{canc}
\{\Delta (g)(1 \otimes h)\, |\, g, h \in H\}^{cls} =
H \otimes_{\mathrm{min}} H =
\{(g \otimes 1)\Delta (h)\, |\, g, h \in H\}^{cls},
\]
we call $(H, \Delta)$ a compact quantum group.
Here \emph{cls} stands for ``closed linear span''. 
If $H$ is commutative, then $H=C(G)$ and $\Delta(f)(g_1,g_2)=f(g_1g_2)$,
where $G$ is a compact Hausdorff topological group. Note that in this case the injectivity of the coproduct follows
from the cancellation properties.

\subsubsection{Free actions \cite{p-p95,e-da00}} \label{ssec:actions}
Let $A$ be a unital C$^*$-algebra, and let $(H, \Delta)$ be a compact quantum group.
We say that the quantum group $(H, \Delta)$ \emph{acts} on $A$ iff
we are given an injective  $*$-homomorphism
 $\delta  : A \to A \otimes_{\mathrm{min}} H$ (called a \emph{coaction}) satisfying:
\begin{enumerate}
\item
$(\delta  \otimes \id) \circ \delta  = (\id \otimes \Delta) \circ \delta$  (coassociativity),
\item
$\{\delta (a)(1 \otimes h) \,|\, a \in A, h \in H\}^{cls} = A \otimes_{\mathrm{min}} H$ (counitality).
\end{enumerate}
We call such an action (coaction) \emph{free} iff
$\{(x \otimes 1)\delta (y)\, |\, x, y \in A\}^{cls} = A \otimes_{\mathrm{min}} H$.
If $A$ and $H$ are commutative, then $\delta(f)(x,g)=f(xg)$, where $X\times G\ni (x,g)\mapsto xg\in X$
is a free continuous action of a compact Hausdorff group $G$ on a compact Hausdorff space~$X$.
Note that in this case the injectivity of the coaction is equivalent to its counitality for all continuous (not necessarily free) actions
$X\times G\to X$.

\subsubsection{Equivariant join \cite{dhh15,bdh16}}
There is an obvious definition of the join of unital \mbox{C*-algebras}
\begin{equation}\label{stan}
A {\circledast } B :=
\left\{f\in C\big([0,1],A \underset{\mathrm{min}}{\otimes} B\big) \,
\Big|\, f(0) \in \mathbb{C}\otimes B,\; f(1)\in A\ot \mathbb{C}\right\}
\end{equation}
dualizing the join of compact Hausdorff spaces: $C(X)\circledast C(Y)=C(X*Y)$.
In particular, we can take $B=H$, where $(H,\Delta)$ is a compact quantum group acting on~$A$.

However, to make this construction equivariant as in the classical case, we need to modify the above
definition to become:
\begin{definition}\label{gjoin}
Let $(H,\Delta)$ be a compact quantum group acting on a unital \mbox{C$^*$-algebra~$A$}
 via $\delta:A\rightarrow A\otimes_{\mathrm{min}}H$.
We call the unital C*-algebra
\begin{equation}\label{equiv}
A \overset{\delta}{\circledast } H :=
\left\{f\in C\big([0,1],A \underset{\mathrm{min}}{\otimes} H\big) \,
\Big|\, f(0) \in \mathbb{C}\otimes H,\; f(1)\in\delta(A)\right\}
\end{equation}
the \emph{equivariant noncommutative join} of $A$ and $H$.
\end{definition}
\noindent
The coaction
\[
(\id\otimes\Delta)\circ\;\colon C\big([0,1],A \underset{\mathrm{min}}{\otimes} H\big)\longrightarrow
C\big([0,1],A \underset{\mathrm{min}}{\otimes} H\big)\underset{\mathrm{min}}{\otimes} H
\]
induces a free action of $(H,\Delta)$ on $A\circledast^\delta H$.

\begin{remark}\label{remark}
It is important to note that, when $H$ is commutative, the standard join C*-algebra~\eqref{stan} with the diagonal coaction
(which is an algebra homomorphism due to the commutativity of~$H$) is equivariantly isomorphic to the equivariant join C*-algebra~\eqref{equiv}
(see \cite[Section~1.3]{hnpsz18}).
\end{remark}

\subsubsection{The Peter-Weyl subalgebra and associated modules \cite{bdh16}}

Let $A$ be a unital C*-algebra with a free action of a compact quantum group $(H,\Delta)$ given
by a coaction \mbox{$\delta: A\to A\otimes_{\rm min}H$}.
We define the
\emph{Peter-Weyl subalgebra} of $A$  as
\[
\cP_H(A):=\{\,a\in A\;| \;\delta(a)\in A\otimes\cO (H)\,\}.
\]
Here $\Pol(H)$ is the Hopf algebra
spanned by the matrix coefficients of irreducible unitary corepresentations
of~$H$.
Using the coassociativity of $\delta$, one can  check that $\cP_H(A)$ is a right $\cO (H)$-comodule algebra.
For $(A,\delta)=(H,\Delta)$ we have $\P_H(H)=\cO(H)$.

Now we can define the cotensor product
\[\label{pwcot}
\mathcal{P}_H(A)\Box V:=\{t\in \mathcal{P}_H(A)\otimes V\;|\;(\delta\otimes\id)(t)=
(\id\otimes{}_V\Delta)(t)\},
\]
where ${}_V\Delta\colon V\to \mathcal{O}(H)\otimes V$ is a left coaction rendering
$V$ a left $\mathcal{O}(H)$-comodule. We know that, if $\dim V<\infty$,
then the cotensor product $\mathcal{P}_H(A)\Box V$ is finitely generated projective as a left
module over the fixed-point subalgebra $A^{\mathrm{co}H}:=\{a\in A\;|\;\delta(a)=a\otimes 1\}$.

Next, we unravel  the meaning of the associated module $\mathcal{P}_H(A)\Box V$  for commutative \mbox{C*-algebras}
$A=C(X)$ and $H=C(G)$ (the classical setting). Let
$\varrho\colon G\to GL(V)$ be a  representation of $G$ on a finite-dimensional vector space $V$.
We view
$V$ as a left $\Pol(C(G))$-comodule via the coaction ${}_V\Delta\colon V\to \Pol(C(G))\otimes V$
given by the formula
$({}_V\Delta(v))(g)=\varrho(g^{-1})v$.
Note that the $C(X/G)$-module $C_G(X,V)$ of all continuous \mbox{$G$-equivariant} functions from $X$ to $V$
is naturally isomorphic with the $C(X/G)$-module $\Gamma(X\times^GV)$ of all continuous sections
of~$X\times^GV$. Here $G$-equivariance means
\[
\forall\;x\in X,\, g\in G:\;f(xg)=\varrho(g^{-1})(f(x)).
\]
Combining this with \cite[Corollary~4.1]{bdh16}, we can naturally identify
$\mathcal{P}_{C(G)}(C(X))\Box V$ with $\Gamma(X\times^GV)$ as left $C(X/G)$-modules.

\subsubsection{Induced actions of compact quantum subgroups}\label{two}
Given a compact quantum group $(H,\Delta)$, by a \emph{compact (closed) quantum subgroup} we mean a compact quantum group
$(K,\Delta)$ together with a surjective  $*$-homomorphism $\pi\colon H\to K$ respecting the coproducts.

\begin{remark}\label{cancel}
Note that the cancellation properties \eqref{canc} for $(K,\Delta)$ follow from 
the surjectivity of $\pi\colon H\to K$ combined with the cancellation properties  for $(H,\Delta)$. 
In other words, if a closed quantum subsemigroup of  a compact quantum group has injective coproduct, it is
a compact quantum group.
\end{remark}

The *-homomorphism~$\pi$ defines by restriction-corestriction a homomorphism\linebreak
 \mbox{$\Pol(H)\to\Pol(K)$} of Hopf $*$-algebras. It is always surjective as otherwise there would be a non-zero element of $\Pol(K)$ orthogonal to
$\pi(\Pol(H))$, which contradicts the density of $\pi(\Pol(H))$ in~$K$.
On the other hand, the converse is not always true, that is, a surjective homomorphism $\Pol(H)\to\Pol(K)$ of
Hopf $*$-algebras not always can be extended to a homomorphism $H\to K$ of C$^*$-algebras. (For instance,
the counit of the group-ring $*$-Hopf algebra $\mathbb{C}F_2$, where $F_2$ is the free group on two generators,
cannot be extended to a $*$-homomorphism $C^*_rF_2\to\mathbb{C}$, where $C^*_rF_2$ is the reduced convolution
C$^*$-algebra of~$F_2$.)
However, such an extension exists, for example,
if $H$ coincides with its universal form 
defined as the enveloping C*-algebra of~$\Pol(H)$.

Assume we are given a coaction $\delta\colon A \to A \otimes_{\mathrm{min}} H$. Then the map
\[
\delta':=(\id\otimes\pi)\circ\delta: A \longrightarrow A \otimes_{\mathrm{min}} K
\]
satisfies the coassociativity and counitality conditions ((1) and (2) in \ref{ssec:actions}), but it may fail to be injective.
To remedy this problem, take $\overline{A}:=A/\ker\delta'$. Then $\delta'$ induces an injective map
$
\overline{\delta}: \overline{A} \rightarrow \overline{A} \otimes_{\mathrm{min}} K
$
defining an action of $(K,\Delta)$ on $\overline{A}$. Indeed, suppose $\ker\overline{\delta}$ 
contains the class  of an element $a\in A$. Then, using the commutativity of the diagram
\[
\xymatrix{
A \ar[r]^{\delta'\ \ \ \ \ }\ar[d] & A \otimes_{\mathrm{min}} K\ar[d]\\
\overline{A}\ar[r]_{\overline{\delta}\ \ \ \ \ } & \overline{A}\otimes_{\mathrm{min}} K,
}
\]
for any bounded linear functional $\omega$ on $K$, we get 
\[
(\id\otimes\omega)(\delta'(a))\in\ker(A\longrightarrow\overline{A})=\ker\delta'.
\]
Since $\omega$ is arbitrary, we conclude that $(\delta'\ot\id)(\delta'(a))=0$. 
Now, taking advantage of the coassociativity of $\delta'$ and the injectivity of $\Delta$, we infer that
$\delta'(a)=0$, so the class of $a$ in~$\overline{A}$ is zero. 
This proves the injectivity of $\overline{\delta}$. Other desired properties of $\overline{\delta}$ are immediate.
We say that the coaction $\overline{\delta}: \overline{A} \rightarrow \overline{A} \otimes_{\mathrm{min}} K$ is \emph{induced by the coaction} 
$\delta\colon A \to A \otimes_{\mathrm{min}} H$.

It is straightforward to verify that the construction of induced actions of closed quantum subgroups is functorial:
Let $A$, $B$ and $C$ be unital C*-algebras equipped with an action of a compact quantum group $(H,\Delta)$. If $f\colon A\to B$
and $g\colon B\to C$ are equvariant \mbox{*-homomorphisms}, then, for any closed quantum subgroup $(K,\Delta)$, we obtain the natural
induced equivariant maps \mbox{$\overline{f}\colon \overline{A}\to \overline{B}$}
and $\overline{g}\colon \overline{B}\to \overline{C}$ such that $\overline{g}\circ\overline{f}$ is the induced map 
$\overline{g\circ f}\colon \overline{A}\to \overline{C}$.

Let us finally note that if $K$ is commutative, then  $\overline{A}=A$. The reason is that in this case we always
have the counit map $\eps\colon K\to\mathbb{C}$, which satisfies  $(\id\ot\eps)\circ\delta'=\id$,  proving that $\delta'$ is injective.

\section{The noncommutative Borsuk-Ulam conjecture of type I}
\noindent
\begin{conjecture}[\cite{bdh15}]\label{mc}
Let $A$ be a unital  C*-algebra equipped with
a free action  
of a non-trivial compact quantum group~$(H,\Delta)$, and let
$A\circledast^\delta H$ be the equivariant noncommutative join C*-algebra of $A$
 and $H$ with the induced free action of~$(H,\Delta)$.
Then
\begin{equation*}
\text{\large\boldmath$\not\!\exists$}\;\;\text{$(H,\Delta)$-equivariant $*$-homomorphism}\; A\longrightarrow
A\circledast^\delta H.
\end{equation*}
\end{conjecture}

It turns out that whether or not the C*-algebras $A$ and $H$ admit a character has crucial consequences for the conjecture.
Therefore, let us consider the following four cases:
\begin{enumerate}
\item
\emph{$A$ admits a character and $H$ does not admit a character:}
We believe that in this case it is impossible to have a free action of the compact quantum group $(H,\Delta)$ on~$A$.
\item
\emph{$A$ does not admit a character and $H$ admits a character:}
If there exists a \mbox{$*$-ho}\-mo\-mor\-phism $A\rightarrow A\circledast^\delta H$, then we can follow it by the evaluation at $0$
composed with a character on~$H$, which yields a character on~$A$. Hence, in this case, the above conjecture holds for a trivial reason.
(Note that it holds even for the trivial group.) Obvious examples are provided by: the irrational rotation C*-algebra of a noncommutative
torus equipped with the standard free action of the two-torus, the matrix algebra $M_n(\mathbb C)$, $n>1$, equipped with a free
$\mathbb{Z}/n\mathbb{Z}$-action, the Cuntz algebras $\mathcal{O}_n$, $n>1$, with the gauge action of $U(1)$.
\item
\emph{$A$ does not admit a character and $H$ does not admit a character:}
In this case, we can prove the conjecture for $A=H=C^*_rF_n$, where $F_n$ is the free group generated by $n$-elements with $n>1$
(see Theorem~\ref{k-contractibility2}).
\item
\emph{$A$ admits a character and $H$ admits a character:}
This case is the mainstream of current efforts. It contains the torsion case (section below), and the case of finite quantum groups
with $A=H$
(see Theorem~\ref{k-contractibility}).
\end{enumerate}

\subsection{The torsion case}

Recall that, for any compact quantum group $(H,\Delta)$ admitting characters, their convolution product
is given by $\chi_1*\chi_2:=(\chi_1\otimes\chi_2)\circ\Delta$. The map $\chi_1*\chi_2\colon H\to\mathbb{C}$
 is again a character, and the convolution product
is associative due to the coassociativity of the coproduct~$\Delta$. The neutral element for the convolution product is the counit.
As explained in the next paragraph, its existence follows from the existence of a character.
In the commutative case $H=C(G)$, where $G$ is a compact Hausdorff topological group, the characters of $H$ are identified
with points in $G$ via the Gelfand transform, and thus their convolution product can be identified with the group multiplication.

Next, note that, if the set $X_B$ of all characters of a unital C*-algebra $B$ is non-empty, then we can
topologize $X_B$ by the weak*-topology turining it into a compact Hausdorff space, 
and obtain a surjective \mbox{*-ho}\-mo\-mor\-phism $B\to C(X_B)$.
If $B=H$ is the \mbox{C*-algebra} of a compact quantum group, then $X_H$ becomes a compact Hausdorff topological semigroup
with respect to the convolution product. The induced coproduct on $C(X_H)$ makes $(C(X_H),\Delta)$ a closed subsemigroup of~$(H,\Delta)$.
Hence, by Remark~\ref{cancel}, $G_H:=X_H$ is a compact Hausdorff topological semigroup with cancellation properties, so by \cite{h-kh96}
it is a compact Hausdorff group. We call it the \emph{maximal classical subgroup}
of~$(H,\Delta)$.  The counit of $(H,\Delta)$ is the composition of 
\mbox{$H\to C(G_H)$} with the evaluation at the neutral element $\mathrm{ev}_e\colon C(G_H)\to\mathbb{C}$.

\begin{theorem}\label{mainthm}
Let $A$ be a unital  C*-algebra with
a free action
of a  compact quantum group~$(H,\Delta)$, and let
$A\circledast^\delta H$ be the equivariant noncommutative join C*-algebra of $A$
 and $H$ with the induced free action of~$(H,\Delta)$.
Then, if $H$ admits a character $\chi$,
 the following statements are equivalent:
\begin{gather}\label{main}
\text{\large\boldmath$\not\!\exists$}\;\;\text{$(H,\Delta)$-equivariant $*$-homomorphism}\; A\longrightarrow
A\circledast^\delta H,\\
\label{maincone}
\text{\large\boldmath$\not\!\exists$}\;\text{$*$-homomorphism}\; \gamma:A\longrightarrow
\CC A \;
\text{such that $\ev_1\circ\gamma\colon A\to A$ is $H$-colinear}.
\end{gather}
Here $\CC A:=A\circledast \mathbb{C}$ is the cone of~$A$, and $\ev_1$ is the evaluation at~1.
Moreover, these statements are true if $\chi$ is different from the counit but its finite convolution power is the counit.
\end{theorem}

\begin{proof}[Proof of Theorem~\ref{mainthm}]
The proof is based on three lemmas and \cite[Corollary~2.4]{p-b16}.
First, let $(K,\Delta)$ be a closed quantum subgroup of $(H,\Delta)$.
Then, by the functoriality of the induced action (see Section~\ref{two}), an equivariant *-homomorphism 
$\varphi\colon A\to A\circledast^\delta H$
induces  an equivariant *-homomorphism $\overline{\varphi}\colon\overline{A}\to\overline{A {\circledast}^{\delta} H}$. 
Since $\delta'=(\id\ot\id\ot\pi)\circ(\id\ot\Delta)$ and $\pi\colon H\to K$ intertwines the coproducts, for any $p\in\ker(\delta')$, we obtain 
\[
0=\Big(\id\ot\big((\pi\ot\pi)\circ\Delta\big)\Big)(p)=(\id\ot\Delta)\big((\id\ot\pi)(p)\big).
\]
Now it follows from the injectivity of $\Delta$ that 
\[\label{forall}
\forall\;p\in\ker(A\circledast^\delta H\longrightarrow\overline{A\circledast^\delta H})\colon(\id\ot\pi)(p)=0. 
\]

Furthermore, using the naturality of the equivariant join construction in both variables, it is straightforward to verify that
the quotient maps $A\to\overline{A}$ and $\pi\colon H\to K$ yield a \mbox{*-homomorphism}
\[
F\colon A {\circledast}^{\delta} H\longrightarrow\overline{A} {\circledast}^{\overline{\delta}} K.
\]
Since, by \eqref{forall}, $\ker(A {\circledast}^{\delta} H\to\overline{A {\circledast}^{\delta} H})\subseteq\ker(F)$, the map $F$
 factors through $\overline{A {\circledast}^{\delta} H}$ giving a \mbox{*-homomorphism}
\[
\overline{F}\colon\overline{A {\circledast}^{\delta} H}\longrightarrow\overline{A} {\circledast}^{\overline{\delta}} K.
\]
It is easy to check that $\overline{F}$ is equivariant. Thus we obtain an equivariant *-homomorphism 
$\overline{F}\circ\overline{\varphi}\colon\overline{A}\to \overline{A} {\circledast}^{\overline{\delta}} K$.
This proves our first lemma:
\begin{lemma}\label{l1}
Let $(H,\Delta)$ be an arbitrary compact quantum group with a closed quantum subgroup~$(K,\Delta)$. Then \eqref{main} holds
 if
$$
\text{\large\boldmath$\not\!\exists$}\;\;\text{$(K,\Delta)$-equivariant $*$-homomorphism}\;
\overline{A}\longrightarrow \overline{A}\circledast^{\overline{\delta}} K.
$$
\end{lemma}

We want to apply this lemma by taking $(K,\Delta)$ to be a finite classical  subgroup of~$(H,\Delta)$.
In this case, $\overline{A}=A$ (see Section~\ref{two}).
To this end,
observe that the character $\chi\neq\varepsilon$ generates a non-trivial cyclic subgroup of~$(H,\Delta)$.
Indeed, by assumption, there exists $n>1$ such that $\chi^{*n}=\varepsilon$. Hence $X:=\{\varepsilon, \chi,\ldots,\chi^{*(n-1)}\}$
is a non-trivial cyclic subgroup of the maximal classical subgroup~$G_H$, and consequently of~$(H,\Delta)$.
Now, combining Lemma~\ref{l1} for $K=C(X)$ with \cite[Corollary~2.4]{p-b16} and Remark~\ref{remark},
we infer that \eqref{main} holds as claimed.

Next, we will prove the equivalence of \eqref{main} and \eqref{maincone} without assuming any properties of the character~$\chi$.
 If  $\gamma:A \to \CC A$ is as in~\eqref{maincone},
then
$\phi:=(\gamma\ot\id)\circ \delta $
descends to a map as in~\eqref{main}.
Indeed, $\phi (A)\subseteq A\circledast^\delta H$ by the properties of $\gamma$. To check the equivariance of~$\phi$,
we use the identification
\[\label{ident}
C\big([0,1],A\otimes_ \mathrm{min} B\big) =
C([0,1])\otimes A \otimes_ \mathrm{min} B
\]
 and the coassociativity of the coaction $\delta$ to compute:
\begin{align}
(\id\ot\id\ot\Delta)\circ\phi&=(\id\ot\id\ot\Delta)\circ(\gamma\otimes\id)\circ\delta \nonumber\\
&=
(\gamma\otimes\id\otimes\id)\circ (\delta\otimes\id)\circ\delta \nonumber\\
&=
(\phi\otimes\id)\circ\delta .
\end{align}
Thus we have proved:
\begin{lemma}\label{l3a}
Let $(H,\Delta)$ be any compact quantum group. Then \eqref{main} follows from \eqref{maincone}.
\end{lemma}

For the converse implication, we need the existence of counit, which follows from the assumed existence
of a character.
Now, using again the identification~\eqref{ident}, if $\phi$ is as in \eqref{main},
then
\[
\gamma:= (\id\ot\id\ot\veps) \circ \phi
\]
descends to a map as in~\eqref{maincone}.
Indeed, $\gamma (A)\subseteq \CC A$ by the properties of $\phi$. To see the equivariance of $\ev_1\circ\gamma$,
we use  the fact that $\gamma$ takes values in the cone of $A$, whence  for any $a\in A$ there is $a'\in A$ such that
\[\label{aprime}
((\ev_1\otimes\id\otimes\id)\circ \phi)(a)=\delta(a').
\]
 Taking advantage of this formula
and the counitality of the coaction~$\delta$, we compute:
\begin{align}
(\delta\circ(\ev_1\otimes\id)\circ\gamma)(a)&=(\delta\circ(\ev_1\otimes\id)\circ  (\id\ot\id\ot\veps) \circ \phi)(a) \nonumber\\
&=(\delta\circ(\id\otimes\veps)\circ(\ev_1\otimes\id\otimes\id)\circ \phi)(a) \nonumber\\
&=(\delta\circ(\id\otimes\veps)\circ \delta)(a')\nonumber\\
&=\delta(a').
\end{align}
Much in the same way, using the equivariance of $\varphi$, the counitality of the comultiplication~$\Delta$,
and \eqref{aprime}, we obtain:
\begin{align}
&((\ev_1\otimes\id\otimes\id)\circ(\gamma\otimes\id)\circ\delta)(a)\nonumber\\
&=
((\ev_1\otimes\id\otimes\id)\circ( (\id\ot\id\ot\veps\otimes\id) \circ (\phi\otimes\id)\circ\delta)(a)\nonumber\\
&=
((\ev_1\otimes\id\ot\id)\circ  (\id\ot\id\ot\veps\ot\id)\circ(\id\otimes\id\ot\Delta)\circ\phi)(a) \nonumber\\
&=
((\ev_1\otimes\id\ot\id)\circ\phi)(a) \nonumber\\
&=\delta(a').
\end{align}
This establishes the equivariance of  $\ev_1\circ\gamma$ and proves:
\begin{lemma}\label{l3b}
Let $(H,\Delta)$ be a compact quantum group admitting a character. Then \eqref{main} implies~\eqref{maincone}.
\end{lemma}

Summarizing,
assuming that $H$ has a non-trivial torsion character we proved that \eqref{main} is true, and assuming that $H$ has any character
we showed that \eqref{main} and \eqref{maincone} are equivalent.
\end{proof}

\subsection{Non-contractibility of unital C*-algebras admitting free actions}

There are many different flavours of the concept of the contractibility of C*-algebras. We choose the notion of contractibility that
recovers the standard notion of contractibility of compact Hausdorff spaces in the commutative setting, and that best fits
the context of the noncommutative Borsuk-Ulam-type conjectures.
\begin{definition}
Let $A$ be a unital C*-algebra. We call it \emph{unitally contractible} iff
$$
\text{\large\boldmath$\exists$}\;\;\text{unital $*$-homomorphism}\; \gamma:A\longrightarrow
\CC A\;
\text{such that $\ev_1\circ\gamma= \id_A$.
}
$$
\end{definition}
\noindent
Note that $\gamma$ plays the role of a noncommutative retraction. Observe also that, if $A$ is unitally contractible,
then $K_0(A)=\mathbb{Z}[1]\cong\mathbb{Z}$ and $K_1(A)=0$.

Now, since the identity map is always equivariant, as a corollary of~\eqref{maincone},  we infer:
\begin{corollary}\label{coneB}
Let $A$ be a unital  C*-algebra with a free action of a  compact quantum group~$(H,\Delta)$.
Then, if $H$ admits a character different from the counit whose finite convolution power is the counit,
the C*-algebra $A$ is \emph{not} unitally contractible.
\end{corollary}
\noindent
Hence, if $A$ is unitally contractible, no compact quantum group with classical torsion can act freely on~$A$.
By contraposition, if a compact quantum group with classical torsion acts freely on $A$, then $A$ cannot be  unitally contractible.

For instance, since the restriction to $\Z/2\Z$ of the gauge action of $U(1)$ on the Toeplitz algebra $\mathcal{T}$ is free,
$\mathcal{T}$ is not unitally contractible. This fact is interesting because it cannot be deduced immediately from
the K-theory of $\mathcal{T}$,
as it coincides with the K-theory of the classical disc~$D$, which is contractible. Thus the quantum disc
(we think of $\mathcal{T}$ as the C*-algebra of the quantum disc) is not contractible, unlike its classical counterpart.

A way to understand this phenomenon, which we owe to Jingbo Xia, is as follows.
Unlike the function C*-algebra $C(D)$, the Toeplitz algebra $\mathcal{T}$ admits non-trivial  projections as its elements.
We can faithfully represent $\mathcal{T}$ on an infinite-dimensional Hilbert space $l^2(\mathbb{N})$ by sending
its generating isometry $s$ to the unilateral shift~\cite{c-la67}. Then $1-ss^*$ becomes a rank-one projection.
If $\mathcal{T}$ were unitally contractible, we would have a norm-continuous map $p$ from the interval $[0,1]$
to the algebra of bounded operators $B(l^2(\mathbb{N}))$ such that $p(t)$ is a projection for every $t\in[0,1]$,
the projection $p(0)$ is a scalar multiple of the identity operator, and $p(1)=1-ss^*$. However, this contradicts the continuity of the rank,
as the rank of $p(0)$ is zero or infinity and the rank of $p(1)$ is one.

Furthermore, it is worth noticing that Corollary~\ref{coneB} can be interpreted as
a noncommutative generalization of the Brouwer fixed-point theorem.
Indeed, in the particular case when $A=C(S^n)$ and $H=C(\Z/2\Z)$,
this corollary is well known to be equivalent to the Brouwer fixed-point theorem for
continuous maps $B^n\to B^n$ between balls. (See \cite{d-l16} for details.)

Let us now focus on an important special case: $(A,\delta)=(H,\Delta)$. Then the property \eqref{main}  boils down to the non-trivializability
(in the sense of \cite[Definition~3.1]{bdh15}) of the compact quantum 
principal $H$-bundle~$H\circledast^\Delta H$.
In the classical setting $H=C(G)$, this is equivalent to
 the well-known~\cite{h-b79} non-contractibility of non-trivial
compact Hausdorff groups. However, in the quantum case, we knew the non-triviality of $H\circledast^\Delta H$ only
for the quantum
group $SU_q(2)$~\cite{bdh15}.
Now the non-contractibility result follows for the $q$-deformations of any non-trivial compact connected semisimple Lie group from
Theorem~\ref{secondmain} below, and, more generally, for any compact quantum group whose C*-algebra admits a character
different from the counit but whose finite convolution power is the counit, from Theorem~\ref{mainthm}.

A \emph{finite quantum group} is a compact quantum group $(H,\Delta)$  such that $\dim H<\infty$. 
Our aim now is to prove  Conjecture~\ref{mc} in the special case $(A,\delta)=(H,\Delta)$ when $(H,\Delta)$ is a finite quantum group.
To achieve this, we first show that the unital contractibility of $H$ implies the property~\eqref{main} rendering these two properties equivalent,
as in the classical case. This way we recover the
classical mechanism relating the non-triviality of the principal $G$-bundle $G*G$ with the non-contractibility of~$G$. We begin by proving
a general Hopf-algebraic fact.
\begin{proposition}\label{hopflemma}
Let $\mathcal{H}$ be a Hopf algebra over any field. Then
any right-colinear algebra homomorphism  \mbox{$\phi:\mathcal{H}\to \mathcal{H}$} is invertibile.
The inverse is given by:
\begin{equation}\label{inv}
\phi^{-1} = m\circ\big((\veps\circ\phi\circ S)\otimes\id\big)\circ\Delta .
\end{equation}
Here $\veps$, $S$ and $\Delta$ stand respectively for the counit, the antipode, and the coprodcut.
\end{proposition}
\begin{proof}
Note first that applying $\veps\otimes\id$ to the colinearity equation $\Delta\circ\varphi=(\varphi\otimes\id)\circ\Delta$,
we obtain
\begin{equation}\label{inv1}
\phi = \big((\veps\circ\phi)\otimes\id\big)\circ\Delta = (\veps\circ\phi)*\id,
\end{equation}
where $*$ denotes the convolution product.
We also write the map on the right hand side
of \eqref{inv} as
\[(\veps\circ\phi\circ S)*\id.
\]
Furthermore, for any two linear functionals $\alpha$ and $\beta$ on $\mathcal{H}$, it is straighforward to verify that
\[\label{conv}
(\alpha*\id)\circ(\beta*\id)= \alpha*\beta*\id.
\]
Finally, since $\veps\circ\phi\circ S$  is the convolution inverse of $\veps\circ\phi$, we end the proof.
\end{proof}

Since for a finite quantum group $(H,\Delta)$
the C*-algebra $H$ is also a Hopf algebra, we immediately conclude:
\begin{corollary}\label{equivinv}
Let  $(H,\Delta)$ be a finite quantum group. Then any right-equivariant  \mbox{*-homomorphism}
$\varphi:H\to H$ of  is invertibile.
\end{corollary}

As any finite quantum group always admits a character, we can  combine Lemma~\ref{l3a}
with Lemma~\ref{l3b} to infer  that \eqref{main} and \eqref{maincone} are equivalent. Furthermore, it follows from
 Corollary~\ref{equivinv} that the unital non-contractibility of $H$ is
equivalent to~\eqref{maincone}. Indeed, $\gamma\circ(\ev_1\circ\gamma)^{-1}\colon A\to\mathcal{C}A$
satisfies $\ev_1\circ(\gamma\circ(\ev_1\circ\gamma)^{-1})=\id$. Consequently, \eqref{main} is equivalent to
the unital non-contractibility of~$H$, which recovers the classical mechanism referred to earlier on.

Next, as any finite-dimensional C*-algebra $A$ is a direct sum of matrix algebras,
$K_0(A)=\mathbb{Z}[1]$ only for $A=\mathbb{C}$. Therefore, if $(H,\Delta)$ is a non-trivial finite quantum group,
then $H$ is not unitally contractible. Thus we have arrived at:
\begin{theorem}\label{k-contractibility}
If a compact quantum group $(H,\Delta)$ is non-trivial and finite,
then Conjecture~\ref{mc} holds for
$(A,\delta)=(H,\Delta)$.
\end{theorem}
\noindent
Note that the above theorem does not follow from Theorem~\ref{mainthm} as there exists
 a finite quantum group whose only character is the counit.

\section{The  noncommutative Borsuk-Ulam conjecture of type II}
\noindent
While the noncommutative Borsuk-Ulam conjecture of type~I remains open as stated, there are counterexamples to
the  noncommutative Borsuk-Ulam conjecture of type~II \cite[Theorem~2.2, Theorem~2.5]{cp16}.
However, assuming that a C*-algebra acted upon admits a character, not only do we remove counterexaples, but also we prove
that the type-II conjecture follows from the type-I conjecture. Note, however, that  we do not know when this assumption is necessary
(see~\cite[Theorem~2.6]{cp16}) for the type-II conjecture to hold.
For instance, one can also remove the counterexamples by assuming that both $A$ and $H$ do not admit a character. This assumption
leads to a modification of the type-II conjecture that might \emph{not} follow from the type-I conjecture.
Nevertheless, it is clear now that the type-I conjecture supersedes
the type-II conjecture. Hence, for the sake of simplicity, in the forthcoming papers,
 we shall refer to the type-I conjecture as the ``noncommutative Borsuk-Ulam-type'' conjecture.

Let us now show the aforementioned implication.
If $A$ has a character $\alpha$, then an \mbox{$(H,\Delta)$-equivariant} *-homomorphism as in \eqref{main2}
composed with the $(H,\Delta)$-equivariant  *-ho\-mo\-mor\-phism
$(\alpha\otimes \id)\circ\delta$ would provide an $(H,\Delta)$-equivariant *-homomorphism  as in~\eqref{main}.
Thus the following modification of \cite[Conjecture~2.3 (type~II)]{bdh15}  is a corollary of Conjecture~\ref{mc}:
\begin{conjecture}\label{c2}
Let $A$ be a unital  C*-algebra equipped with
a free action  
of a non-trivial compact quantum group~$(H,\Delta)$, and let
$A\circledast^\delta H$ be the equivariant noncommutative join C*-algebra of $A$
 and $H$ with the induced free action of~$(H,\Delta)$.
Then, if $A$ admits a character,
\begin{equation*}
\text{\large\boldmath$\not\!\exists$}\;\;\text{$(H,\Delta)$-equivariant $*$-homomorphism}\; H\longrightarrow
A\circledast^\delta H.
\end{equation*}
In other words, the compact quantum principal bundle given by $A\circledast^\delta H$ is not trivializable (see
\cite[Definition~3.1]{bdh15}).
\end{conjecture}

\subsection{The torsion case}

The following result is an immediate corollary of Theorem~\ref{mainthm}.
It proves Conjecture~\ref{c2}
under an additional assumption.
\begin{corollary}\label{corollarytype2}
Let $A$ be a unital  C*-algebra with
a free action
of a compact quantum group~$(H,\Delta)$, and let
$A\circledast^\delta H$ be the equivariant noncommutative join C*-algebra of $A$
 and $H$ with the induced free action of~$(H,\Delta)$.
Then, if $H$ admits a character different from the counit whose finite convolution power is the counit,
 and $A$ admits any character,
\begin{equation}\label{main2}
\text{\large\boldmath$\not\!\exists$}\;\;\text{$(H,\Delta)$-equivariant $*$-homomorphism}\; H\longrightarrow
A\circledast^\delta H.
\end{equation}
\end{corollary}

\subsection{A K-theoretic variant}

The point of this section is to prove, under a different assumption on a  compact quantum group $(H,\Delta)$,
 a stronger claim than~\eqref{main2}. Not only do we prove that a certain compact quantum principal bundle is not
trivializable, but also we show that there exists an associated noncommutative vector bundle that is not stably trivial.
Our proof is based on the observation
that  the unreduced suspension of a compact quantum group $(H,\Delta)$ is still a
gluing of its cones over the compact quantum group, so the reduced even K-group of the suspension can be studied via a Milnor connecting 
homomorphism relating it to~$K_1(H)$. This mechanism was already used in \cite[Section~2.3]{dhhmw12}, and
reveals itself through the isomorphism $-\beta^{-1}$ (playing the role of the Milnor connecting homomorphism) 
in the proof of Lemma~\ref{lem:a-mf65}.

\begin{theorem}\label{thm2}
Let $(H,\Delta)$ be a compact quantum group admitting a 
finite-dimensional representation $V$ whose $K_1$-class is not trivial, and 
let $A$ be a  unital C*-algebra   equipped with a
free coaction $\delta\colon A\to A\otimes_{\mathrm{min}} H$. Then, if $(A,\delta)=(H,\Delta)$ or
$A$  admits a character, the associated finitely generated projective left
$(A\circledast^\delta H)^{\mathrm{co}\, H}$-module
\[\label{ass2}
\mathcal{P}_{H}\Big(A\overset{\delta}{\circledast} H\Big)\Box V
\]
is \emph{not} stably free.
\end{theorem}

\begin{proof}
As in the classical setting, we first
reduce the theorem to the case $H{\circledast}^{\Delta} H$. We use a character  $\alpha$ on $A$
to define an $H$-equivariant $*$-homomorphism
$ (\alpha\otimes \id)\circ\delta :A\to H$, which induces an $H$-equivariant $*$-homomorphism
\[
f: A\overset{\delta}{\circledast} H\longrightarrow H\overset{\Delta}{\circledast} H.
\]
Due to the equivariance, the $*$-homomorphism $f$ restricts and corestricts to a $*$-homomor\-phism
between fixed-point subalgebras:
\[
f|: \Big(A\overset{\delta}{\circledast} H\Big)^{\mathrm{co}\, H}
\longrightarrow
\Big(H\overset{\Delta}{\circledast} H\Big)^{\mathrm{co}\, H}.
\]
Now we can apply \cite[Theorem 2.1]{hm18}) to obtain
\begin{gather*}
(f|)_*: K_0\Big(\big(A\overset{\delta}{\circledast} H\big)^{\mathrm{co}\, H}\Big)
\longrightarrow
 K_0\Big(\big(H\overset{\Delta}{\circledast} H\big)^{\mathrm{co}\, H}\Big),\\
(f|)_*\Big(\Big[\mathcal{P}_{H}\big(A\overset{\delta}{\circledast} H\big)\Box V\Big]\Big)=
\Big[\mathcal{P}_{H}\big(H\overset{\Delta}{\circledast} H\big)\Box V\Big].
\end{gather*}
Hence, if the module on the right-hand side is not stably free,
then the module on the left-hand side is also not stably free. Thus we have reduced the theorem to the 
case $H{\circledast}^{\Delta} H$, as desired.

To prove the theorem in this case, recall first that our fixed-point subalgebra is now the suspension C*-algebra of~$H$:
\[
\Big(H\overset{\Delta}{\circledast} H\Big)^{\mathrm{co}\, H}=\left\{f\in C([0,1],H)\mid f(0),f(1)\in \C\right\}=:\Sigma H.
\]
A key observation is that $\Sigma H$ is the pullback C*-algebra of the cone C*-algebras of $H$ over $H$ (e.g., see
\cite[(2.26)]{dhhmw12}). Thus it becomes a special
case of the following general setting. 

Let $B$ be the pullback C*-algebra of the  diagram 
\[\nonumber
\xymatrix{
B \ar[r]^{\mathrm{pr}_2}\ar[d]_{\mathrm{pr}_1} & B_2\ar[d]^{\pi_2}\\
B_1\ar[r]_{\pi_1} & A
}
\]
in the category of unital C*-algebras. Assume that $\pi_2$ is surjective. Then we obtain the Mayer-Vietoris six-term exact sequence in K-theory:
\[\label{MV}
\xymatrixcolsep{5pc}\xymatrix{
K_0(B) \ar[r]^{(\mathrm{pr}_1,\,\mathrm{pr}_2)_*\qquad} & K_0(B_1)\oplus K_0(B_2) \ar[r]^{\qquad\pi_{1\,*}-\pi_{2\,*}} & K_0(A) \ar[d]\\
K_1(A)\ar[u]^{\partial} & K_1(B_1)\oplus K_1(B_2) \ar[l] & K_1(B)~. \ar[l]
}
\]
Here $\partial$ is  the Milnor connecting homomorphism  (\cite[p.\ 1191]{hrz13}) defined as follows.
For any  $U\in GL_n(A)$, we set $\partial([U]):=[p_U]-n[1]$, where $p_U$ is the \emph{Milnor idempotent} associated to $U$ 
(see \cite[(2.7)]{dhhmw12}).
\begin{lemma}\label{3.4}
In the above setting, assume that
$K_0(B_i)=\mathbb{Z}[1]\cong\mathbb{Z}$, $K_1(B_i)=0$, $i\in\{1,2\}$. Then, if $0\neq [U]\in K_1(A)$, the $B$-module given
by the associated Milnor idempotent $p_U$ is \emph{not} stably free. Under further assumption that the group homomorphism
given by $\mathbb{Z}\ni1\mapsto [1]\in K_0(A)$ is injective, we also obtain
\[\label{atiyah}
K_0(B)=\partial(K_1(A))\oplus\mathbb{Z}[1]\cong K_1(A)\oplus\mathbb{Z}.
\]
\end{lemma}
\begin{proof}
Under our assumptions, the Mayer-Vietoris exact sequence \eqref{MV} yields the exact sequence:
\[\label{exactness}
0\longrightarrow K_1(A)\stackrel{\partial}{\longrightarrow} K_0(B)\longrightarrow 
\mathbb{Z}[1]\oplus\mathbb{Z}[1]
\longrightarrow K_0(A).
\]
Denote the image of the penultimate map by~$\Gamma$. Then we obtain the short exact sequence:
\[\label{exactness2}
0\longrightarrow K_1(A)\stackrel{\partial}{\longrightarrow} K_0(B)\stackrel{\mathrm{pr}}{\longrightarrow}\Gamma\longrightarrow 0.
\]
Since 
\[
\mathrm{pr}([1])=([1],[1])\in \mathbb{Z}[1]\oplus\mathbb{Z}[1]\cong \mathbb{Z}\oplus\mathbb{Z},
\]
the abelian group $\Gamma$ 
always contains $\mathbb{Z}([1],[1])\cong \mathbb{Z}$. Although it might be bigger,
since any subgroup of a free abelian group is always a free abelian group, $\Gamma$ is a free $\mathbb{Z}$-module, so the sequence 
\eqref{exactness2} splits giving
\[\label{michael}
K_0(B)=\partial(K_1(A))\oplus s(\Gamma)\cong K_1(A)\oplus \Gamma.
\]
Here $s\colon \Gamma\to  K_0(B)$ is a splitting of the corestriction $\mathrm{pr}\colon K_0(B)\to\Gamma$ of 
$(\mathrm{pr}_1,\,\mathrm{pr}_2)_*$.

Suppose now that $[p_U]=m[1]$, i.e.\ the module defined by $p_U$ is stably free. Then, by the exactness of \eqref{exactness2},
\[
0=\mathrm{pr}(\partial([U]))=(m-n)([1],[1]).
\]
Consequently, $m=n$, so $\partial([U])=(m-n)[1]=0$. Hence, remembering the injectivity of~$\partial$, we conclude 
the desired implication:
\[\label{implication}
[U]\neq 0\quad\Rightarrow\quad [p_U]\not\in\mathbb{Z}[1].
\]

Finally, assume now that the group homomorphism
given by $\mathbb{Z}\ni1\mapsto [1]\in K_0(A)$ is injective. Since the rightmost map of \eqref{exactness} sends $(k[1],l[1])$ to $(k-l)[1]$,
 it follows from the assumed injectivity that $(k[1],l[1])$ is in the kernel if and only if~$k=l$. Hence the exactness of \eqref{exactness} implies that 
$\Gamma=\mathbb{Z}([1],[1])\cong\mathbb{Z}$. Taking the splitting $s$ to be given by $s(([1],[1]))=[1]$ turns \eqref{michael} 
to~\eqref{atiyah}.
\end{proof}

We are now ready to finish the proof of the theorem. 
Let $U_V$ be a representation matrix of the corepresentation~$V$. Then $[U_V^{-1}]=-[U_V]\neq 0$. Therefore,
since the pullback diagram of the unreduced suspension $\Sigma H$ satisfies
the assumptions of Lemma~\ref{3.4} needed to prove~\eqref{implication}, we infer that 
$[p_{U_V^{-1}}]\not\in\mathbb{Z}[1]$. 
Combining it with the fact that the module given by $p_{U_V^{-1}}$
is isomorpic
with the associated module $\mathcal{P}_{H}(H{\circledast}^{\Delta} H)\Box V$ (see \cite[Theorem~3.3]{bdh15}, cf.~\cite[Theorem~3.2]{fhmz}), 
we conclude that $\mathcal{P}_{H}(H{\circledast}^{\Delta} H)\Box V$ is not stably free, as claimed.
\end{proof}

\begin{remark}
Let $(H,\Delta)$ be a compact quantum group.
Observe that, if $H$ coincides with its universal form, then it admits a character proving that $K_0(H)$ satisfies the injectivity 
assumption of Lemma~\ref{3.4}. Consequently, \eqref{atiyah} holds, as in the classical case.
\end{remark}

As explained in \cite{bdh15}, the existence of a map as in \eqref{main2} implies the freeness of the associated module
as in~\eqref{ass2}. (To see that the opposite implication is not true, take $A=H=C(\mathbb{Z}/2\mathbb{Z})$.) This brings
us to:
\begin{corollary}\label{3.6}
Under the assumptions of Theorem~\ref{thm2}, the statement \eqref{main2} is true.
\end{corollary}

Our main application of Theorem~\ref{thm2} comes from deformation theory.
Let $G$ be a non-trivial compact connected (but not necessarily simply connected) semisimple Lie group. For $q>0$
consider the quantized universal enveloping algebra $U_q\mfg$, where $\mfg$ is the Lie algebra of~$G$.
For every finite-dimensional representation
$\rho$ of $G$ on~$V$, there is a representation $\rho_q$ of $U_q\mfg$ on the same space
determined uniquely, up to
isomorphism,  by the requirement that the dimensions of the weight spaces remain the same as for
the original representation. The matrix coefficients of these representations span a Hopf *-subalgebra
$\Pol(G_q)$ of $(U_q\mfg)^*$, and the universal C*-completion of $\Pol(G_q)$ defines a compact quantum
group~$(C(G_q),\Delta)$.
For every $(V,\rho)$ as above, the representation $\rho_q$ defines a left $\Pol(G_q)$-comodule
structure on the vector space $V$. We denote the comodule obtained in this way by~$(V_q,\rho_q)$.

\begin{theorem}\label{secondmain}
Let $G$ be a non-trivial compact connected semisimple Lie group, and let $(V,\rho)$ be a representation of $G$ as in Theorem~\ref{Kthm}.
Then, for any $q>0$ and any unital C*-algebra $A$ that admits a character and is equipped with a
free action of $(C(G_q),\Delta_q)$,  the associated finitely generated projective left
$(A\circledast^\delta C(G_q))^{\mathrm{co}\, C(G_q)}$-module
$$
\mathcal{P}_{C(G_q)}\Big(A\overset{\delta}{\circledast} C(G_q)\Big)\Box V_q
$$
is \emph{not} stably free.
\end{theorem}

\begin{proof}
The above assumption on the representation $(V,\rho)$ is equivalent to saying that the $K_1$-class
of its representation matrix is nonzero. By Theorem~\ref{thm2}, in order to prove the
result, it suffices to show that the $K_1$-class of a representation matrix of $(V_q,\rho_q)$ is also not zero, which
follows from the computation of the $K$-theory of
$C(G_q)$ in~\cite[Theorem~6.1]{nt12}.

In more detail, if we fix numbers $0<a<b$ such that $1\in[a,b]$,
then the C*-algebras $C(G_q)$ form a continuous field in a canonical
way \cite{nt11}. Namely, we can identify the spaces $V_q$ with $V$
in such a way that the standard generators of $U_q\mfg$, viewed as
operators on~$V$, depend continuously on $q$. Then the matrix
coefficients in $C(G_q)$ form continuous sections. 
The corresponding section algebra is denoted by $C(G_{[a,b]})$.

By the definition of the continuous field structure, the
representations $V_q$, viewed as elements of $GL_n(C(G_q))$, form a
continuous section, that is, define an element $V_{[a,b]}\in
GL_n(C(G_{[a,b]}))$. As follows from~\cite[Theorem~6.1]{nt12}, the
evaluation maps $C(G_{[a,b]})\to C(G_q)$ define isomorphisms in
$K$-theory. Since $V_{[a,b]}$ is mapped to $V_q$ and $[V_1]=[V]\ne0$,
we conclude that $[V_q]\ne0$ for all $q\in[a,b]$.
\end{proof}

Our second application of Theorem~\ref{thm2} concerns the reduced C*-algebra $C_r^*(F_n)$
of the free group on $n$ elements with $n>1$. Together with the coproduct $\Delta_n$ determined
by sending $g\in F_n$ to
$g\otimes g$, it is an abelian compact quantum group. As the generators of $F_n$ have non-trivial
classes in $K_1(C_r^*(F_n))$ and can be considered as matrices of one-dimensional representations,
we  conclude that Theorem~\ref{thm2} applies to $(H,\Delta)=(C_r^*(F_n),\Delta_n)$.
Hence Corollary~\ref{3.6} yields:
\begin{theorem}\label{k-contractibility2}
For $n>1$, 
Conjecture~\ref{mc} holds for
$(A,\delta)=(C_r^*(F_n),\Delta_n)=(H,\Delta)$.
\end{theorem}
\noindent
Note that the above theorem does not follow from Theorem~\ref{mainthm} as $C_r^*(F_n)$ is a simple C*-algebra for~$n>1$.

\section*{Acknowledgements}
\noindent
This work is part of the project \emph{Quantum Dynamics} sponsored by  EU-grant RISE~691246
and Polish Government grants 317281 and~328941.
This work was also partially supported by the European Research Council under the EU's Seventh Framework
Programme (FP/2007-2013)/ERC Grant Agreement~307663. Finally, Ludwik D\k{a}browski is grateful
for his support at IMPAN provided by
Simons-Fundation grant 346300 and a Polish Government MNiSW 2015--2019 matching fund.
We are very grateful to Alexandru Chirvasitu and  Benjamin Passer for suggestions about the manuscript.
Finally, we are delighted to thank Paul F.\ Baum, Kenny De Commer and Jingbo Xia for extremely
helpful discussions.


\begin{thebibliography}{99}


\bibitem{a-mf65}
M.~F.~Atiyah:
\emph{On the $K$-theory of compact Lie groups},
Topology 4 (1965) 95--99.\vspace*{1mm}

\bibitem{bdh15}
P.~F.~Baum, L.~D\k abrowski, P.~M.~Hajac:
\emph{Noncommutative Borsuk-Ulam-type conjectures},
Banach Center Publications 106 (2015) 9--18.
\vspace*{1mm}

\bibitem{bhms07}
P.~F.~Baum, P.~M.~Hajac, R.~Matthes, W.~Szyma\'nski:
\emph{Noncommutative Geometry Approach to Principal and Associated Bundles},
arXiv:math/0701033. \vspace*{1mm}

\bibitem{bdh16}
P.~F.~Baum, K.~De~Commer, P.~M.~Hajac:
\emph{Free actions of compact quantum group on unital C*-algebras},
\emph{Free actions of compact quantum group on unital C*-algebras}, Documenta Math. 22 (2017) 825--849.
\vspace*{1mm}

\bibitem{b-k33} K.~Borsuk:
\emph{Drei S\"atze \"uber die n-dimensionale euklidische Sph\"are},
Fund. Math. 35 (1933) 177--190.\vspace*{1mm}

\bibitem{cdt} A.~Chirvasitu,  L.~D\k{a}browski, M.~Tobolski:
\emph{The weak Hilbert-Smith conjecture from a Borsuk-Ulam-type conjecture}, 
 arXiv:1612.09567. 
\vspace*{1mm}

\bibitem{cp16} A.~Chirvasitu, B.~Passer:
\emph{Compact group actions on topological and noncommutative joins},
 arXiv:1604.02173v2. To appear.
\vspace*{1mm}

\bibitem{c-la67} L.~A.~Coburn:
\emph{The C*-algebra generated by an isometry},
Bull. Amer. Math. Soc. 73 (1967), 722--726.

\bibitem{d-l16} L.~D\k{a}browski:
\emph{Towards a noncommutative Brouwer fixed-point theorem},
J. Geom. Phys. 105 (2016) 60--65.\vspace*{1mm}

\bibitem{dhh15}
L.~D\k{a}browski, T.~Hadfield, P.~M.~Hajac:
\emph{Equivariant join and fusion of noncommutative algebras},
SIGMA 11 (2015) 082.\vspace*{1mm}

\bibitem{dhhmw12}
L.~D\k{a}browski, T.~Hadfield, P.~M.~Hajac, R.~Matthes, E.~Wagner:
\emph{Index pairings for pullbacks of C*-algebras}, in
\emph{Operator algebras and quantum groups}, 67--84,
Banach Center Publ., 98, Polish Acad. Sci. Inst. Math., Warsaw,  
2012.\vspace*{1mm}

\bibitem{e-da00}
D.~A.~Ellwood:
\emph{A new characterisation of principal actions},
J. Funct. Anal. 173 no.~1 (2000) 49--60.\vspace*{1mm}

\bibitem{fhmz}
C.~Farsi, P.~M.~Hajac, T.~Maszczyk, B.~Zieli\'nski:
\emph{Rank-two Milnor idempotents for the multipullback quantum complex projective plane},
arXiv:1708.04426.\vspace*{1mm}

\bibitem{hm18}
P.~M.~Hajac, T.~Maszczyk:
\emph{Pullbacks and nontriviality of associated noncommutative vector bundles}, arXiv:1601.00021.
To appear in Journal of Noncommutative Geometry.
\vspace*{1mm}

\bibitem{hnpsz18}
P.~M.~Hajac, R.~Nest, D.~Pask, A.~Sims, B.~Zieli\'nski:
\emph{The K-theory of twisted multipullback quantum odd spheres and complex projective spaces},
arXiv:1512.08816. To appear in Journal of Noncommutative Geometry.
\vspace*{1mm}


\bibitem{hrz13}
P.~M.~Hajac, A.~Rennie, B.~Zieli\'nski:
\emph{The K-theory of Heegaard quantum lens spaces},
Journal of Noncommutative Geometry, 7 (2013) 1185--1216.
\vspace*{1mm}

\bibitem{h-l67}
L.~Hodgkin:
\emph{On the $K$-theory of Lie groups},
Topology 6 (1967) 1--36.\vspace*{1mm}

\bibitem{h-b79} B.~Hoffmann:
\emph{A compact contractible topological group is trivial},
Archiv Math. 32, no.~1 (1979) 585--587.\vspace*{1mm}

\bibitem{h-kh96} K.~H.~Hofmann:
\emph{Elements of compact semi-groups},
Charles E. Merill Books Inc., Columbus, Ohio 1996.
\vspace*{1mm}

\bibitem{m-j03}
J.~Matou\v{s}ek:
\emph{Using the Borsuk-Ulam theorem},
Lectures on topological methods in combinatorics and geometry.
Written in cooperation with A.~Bj\: oner and G.~M.~Ziegler. Universitext. Springer-Verlag, Berlin, 2003.
\vspace*{1mm}

\bibitem{nt11}
S.~Neshveyev, L.~Tuset:
\emph{K-homology class of the Dirac operator on a compact quantum group}, Doc. Math. 16 (2011) 767--780.\vspace*{1mm}

\bibitem{nt12}
S.~Neshveyev, L.~Tuset:
\emph{Quantized algebras of functions on homogeneous spaces with Poisson stabilizers},
Comm. Math. Phys. 312 (2012) 223--250.\vspace*{1mm}

\bibitem{p-b16} B.~Passer:
\emph{Free Actions on C*-algebra Suspensions and Joins by Finite Cyclic Groups}
arXiv:1510.04100.
To appear.
\vspace*{1mm}

\bibitem{p-p95} P.~Podle\'s:
\emph{Symmetries of quantum spaces. Subgroups and quotient spaces of quantum $SU(2)$ and $SO(3)$ groups},
 Comm. Math. Phys. 170 (1995) 1--20.\vspace*{1mm}

\bibitem{w-sl98}
S.~L.~Woronowicz:
\emph{Compact quantum groups},
in \emph{Sym\'etries quantiques (Les Houches, 1995)},
845--884, North-Holland, Amsterdam, 1998.

\end{thebibliography}
\end{document}